\documentclass[11pt,reqno]{amsart}
\newtheorem{theorem}{Theorem}[section]
\newtheorem{proposition}[theorem]{Proposition}

\newtheorem{remark}[theorem]{Remark}
\newtheorem{lemma}[theorem]{Lemma}
\newtheorem{example}[theorem]{Example}
\newtheorem{definition}[theorem]{Definition}

\usepackage{amssymb}
\usepackage{graphicx}
\usepackage{amsmath}

\usepackage[all]{xy}
\usepackage[latin1]{inputenc} 
\usepackage{graphicx} 
\usepackage{mathtools}
\numberwithin{equation}{section}

\begin{document}

\title[Lipeng Luo\textsuperscript{1}, Yanyong Hong\textsuperscript{2} and Zhixiang Wu\textsuperscript{3}]{Finite irreducible modules of Lie conformal algebras $\mathcal{W}(a,b)$ and some Schr\"{o}dinger-Virasoro type Lie conformal algebras}
\author{Lipeng Luo\textsuperscript{1}, Yanyong Hong\textsuperscript{2} and Zhixiang Wu\textsuperscript{3}}

\address{\textsuperscript{1}Department of Mathematics, Zhejiang University, Hangzhou, Zhejiang Province,310027,PR China.}
\address{\textsuperscript{2}Department of Mathematics, Hangzhou Normal University,
Hangzhou, 311121, P.R.China.}
\address{\textsuperscript{3}Department of Mathematics, Zhejiang University, Hangzhou, Zhejiang Province,310027,PR China.}

\email{\textsuperscript{1}luolipeng@zju.edu.cn}
\email{\textsuperscript{2}hongyanyong2008@yahoo.com}
\email{\textsuperscript{3}wzx@zju.edu.cn}



\date{\today}
\thanks{ This work was supported by the National Natural Science Foundation of China (No. 11871421, 11501515) and the Zhejiang Provincial Natural Science Foundation of China (No. LQ16A010011).}

\begin{abstract}
Lie conformal algebras $\mathcal{W}(a,b)$ are the semi-direct sums of Virasoro Lie conformal algebra and its nontrivial conformal modules of rank one. In this paper, we first give a complete classification of all finite nontrivial irreducible conformal modules of $\mathcal{W}(a,b)$. It is shown that all such modules are of rank one. Moreover, with a similar method, all finite nontrivial irreducible conformal modules of Schr\"{o}dinger-Virasoro type Lie conformal algebras $TSV(a,b)$ and $TSV(c)$ are characterized.
\end{abstract}

\footnote{The second author is the corresponding author.}
\maketitle

\section{Introduction}\label{intro}
Lie conformal algebra, which was introduced by Kac in \cite{KacV,KacF}, gives an axiomatic description of the operator product expansion (or rather its Fourier transform) of chiral fields in conformal field theory (see \cite{BPZ}). It has been shown that the theory of Lie conformal algebras has close connections to
vertex algebras, infinite-dimensional Lie algebras satisfying the locality property in \cite{KacL} and Hamiltonian formalism in the theory of nonlinear evolution equations (see \cite{BDK}).  Virasoro Lie conformal algebra $Vir$ and current Lie conformal algebra $Cur\mathcal{G}$ associated to a Lie lagebra $\mathcal{G}$  are two important examples of Lie conformal algebras. It was shown in \cite{DK} that $Vir$ and all current Lie conformal algebras $Cur\mathcal{G}$ where $\mathcal{G}$ is a finite dimensional simple Lie algebra exhaust all finite simple Lie conformal algebras. Finite irreducible conformal modules of these simple Lie conformal algebras were classified in \cite{CK} by investigating the representation theory of their extended annihilation algebras. Moreover, the cohomology theory of these Lie conformal algebras was investigated in \cite{BKV}.

For finite non-simple Lie conformal algebras, there are also some developments. One useful method to construct finite non-simple Lie conformal algebras is using the correspondence between formal distribution Lie algebras and Lie conformal algebras. Su and Yuan in \cite{Su-Yuan} investigated two non-simple Lie conformal algebras obtained from Schr\"{o}dinger-Virasoro Lie algebra and the extended Schr\"{o}dinger-Virasoro Lie algebra. Similarly, a class of Lie conformal algebras $\mathcal{W}(b)$ was obtained from the infinite-dimensional Lie algebra $W(a,b)$ which is a semidirect sum of the centerless Virasoro algebra and a intermediate series module $A(a, b)$ (see \cite{Xu-Yue}) and a Schr\"{o}dinger-Virasoro type Lie conformal algebra obtained from a twisted case of the deformative
Schr\"{o}dinger-Virasoro Lie algebra was studied in \cite{Wang-Xu-Xia}. Moreover,  complete classifications of finite irreducible conformal modules of $\mathcal{W}(b)$ and Schr\"{o}dinger-Virasoro Lie conformal algebra given in
\cite{Su-Yuan} were presented in \cite{Wu-Yuan}.
In addition, from the point of view of Lie conformal algebra, Hong in \cite{Hong} presented two classes of Schr\"{o}dinger-Virasoro type Lie conformal algebras $TSV(a,b)$ and $TSV(c)$ and gave a characterization of central extensions, conformal derivations and conformal modules of rank one. Note that $TSV(\frac{3}{2},0)$ is just the Schr\"{o}dinger-Virasoro Lie conformal algebra in \cite{Su-Yuan}
and $TSV(0,0)$ is just the Schr\"{o}dinger-Virasoro type Lie conformal algebra studied in \cite{Wang-Xu-Xia}. In this paper, we plan to give a complete classification of finite irreducible conformal modules of $TSV(a,b)$ and $TSV(c)$. Let $\mathcal{W}(a,b)$ be the semi-direct sum  of $Vir$ and its nontrivial conformal modules of rank one. It is easy to see that $\mathcal{W}(a,b)$ is a subalgebra of $TSV(a,b)$ according to the definition of $TSV(a,b)$. Thus, the representation theory of $TSV(a,b)$ is related with that of $\mathcal{W}(a,b)$. Therefore, we first determine all finite nontrivial irreducible conformal modules of $\mathcal{W}(a,b)$ and give a complete classification of all finite nontrivial irreducible conformal modules of $TSV(a,b)$ and $TSV(c)$. Note that $\mathcal{W}(1-b,0)$ is just the Lie conformal algebra $\mathcal{W}(b)$.

The rest of the paper is organized as follows. In Section 2, we introduce some basic definitions, notations, and related known results about Lie conformal algebras. In Section 3, we first introduce the definition of $\mathcal{W}(a,b)$ and study the extended annihilation algebra of $\mathcal{W}(a,b)$. Then we determine the irreducible property of all free nontrivial rank one $\mathcal{W}(a,b)$-modules over $\mathbb{C}[\partial]$. Finally, we give a complete classification of all finite nontrivial irreducible conformal modules of $\mathcal{W}(a,b)$. In Section 4, we classify all finite nontrivial irreducible conformal modules over Lie conformal algebras $TSV(a,b)$ and $TSV(c)$ by using the results and methods given in Section 3.

Throughout this paper, we use notations $\mathbb{C}$, $\mathbb{C^{*}}$, $\mathbb{Z}$ and $\mathbb{Z^{+}}$ to represent the set of complex numbers, nonzero complex numbers, integers and nonnegative integers, respectively. In addition, all vector spaces and tensor products are over $\mathbb{C}$. In the absence of ambiguity, we abbreviate $\otimes_{\mathbb{C}}$ into $\otimes$.

\section{preliminaries}

 In this section, we recall some basic definitions, notations and related results about Lie conformal algebras for later use. For a detailed description, one can refer to \cite{KacV}.
\begin{definition}
\begin{em}
A \emph {Lie conformal algebra} $\mathcal{R}$ is a $\mathbb{C}[\partial]$-module endowed with a $\mathbb{C}$-linear map from $\mathcal{R}\otimes\mathcal{R}$ to $\mathbb{C}[\lambda]\otimes\mathcal{R}, a\otimes b \mapsto [a_\lambda b]$, called the $\lambda$-bracket, satisfying the following axioms:
\begin{align}
    [\partial a_\lambda b]&=-\lambda[a_\lambda b],\quad  [ a_\lambda \partial b]=(\partial+\lambda)[a_\lambda b] \quad (conformal \  sesquilinearity),\\
    {}[a_\lambda b]&=-[b_{-\lambda-\partial} a]\quad (skew\text{-}symmetry),\\
    {}[a_\lambda [b_\mu c]]&=[[a_\lambda b]_{\lambda+\mu} c]+[ b_\mu [a_\lambda c]]\quad (Jacobi\  identity).
\end{align}
for $a,b,c \in\mathcal{R}$.
\end{em}		
\end{definition}

A Lie conformal algebra $\mathcal{R}$ is called \emph {finite} if $\mathcal{R}$ is finitely generated as a $\mathbb{C}[\partial]$-module. The \emph {rank} of a Lie conformal algebra $\mathcal{R}$, denoted by rank($\mathcal{R}$), is its rank as a $\mathbb{C}[\partial]$-module.

Let $\mathcal{R}$ be a Lie conformal algebra. There is an important Lie algebra associated with it. For each $j \in\mathbb{Z^{+}}$, we can define the  \emph {$j$-th product} $a_{(j)}b$ of two elements $a,b \in\mathcal{R}$ as follows:
\begin{align}
  [a_\lambda b]=\sum_{j\in\mathbb{Z^{+}}}(a_{(j)}b)\frac{\lambda^{j}}{j!}.
\end{align}
Let $Lie(\mathcal{R})$ be the quotient
of the vector space with basis $a_{(n)}$ $(a\in R, n\in\mathbb{Z})$ by
the subspace spanned over $\mathbb{C}$ by
elements:
$$(\alpha a)_{(n)}-\alpha a_{(n)},~~(a+b)_{(n)}-a_{(n)}-b_{(n)},~~(\partial
a)_{(n)}+na_{(n-1)},~~~\text{where}~~a,~~b\in R,~~\alpha\in \mathbb{C},~~n\in
\mathbb{Z}.$$ The operation on $Lie(\mathcal{R})$ is defined as follows:
\begin{equation}\label{106}
[a_{(m)}, b_{(n)}]=\sum_{j\in \mathbb{Z^{+}}}\left(\begin{array}{ccc}
m\\j\end{array}\right)(a_{(j)}b)_{(m+n-j)}.\end{equation} Then
$Lie(\mathcal{R})$ is a Lie algebra and it is called the\emph{ coefficient algebra} of $R$ (see \cite{KacV}).

\begin{definition}
\begin{em}
The \emph {annihilation algebra} associated to a Lie conformal algebra $\mathcal{R}$ is the subalgebra
\begin{align}
  &Lie(\mathcal{R})^{+}=span\{a_{(n)}\ |\ a\in \mathcal{R}, n \in \mathbb{Z^{+}}\},\nonumber
\end{align}
 of the Lie algebra $Lie(\mathcal{R})$. The semi-direct sum of the 1-dimensional Lie algebra $\mathbb{C}\partial$ and $Lie(\mathcal{R})^{+}$ with the action $ \partial (a_{(n)})=-na_{(n-1)}$ is called the \emph {extended annihilation algebra} $Lie(\mathcal{R})^e$.
\end{em}		
\end{definition}

Now, we can introduce the definition of conformal module.
\begin{definition}
\begin{em}
A \emph {conformal module} $M$ over a Lie conformal algebra $\mathcal{R}$ is a $\mathbb{C}[\partial]$-module endowed with a $\mathbb{C}$-linear map $\mathcal{R}\otimes M \rightarrow \mathbb{C}[\lambda]\otimes M, a\otimes v \mapsto a_\lambda v$, satisfying the following conditions:
\begin{align}
    &(\partial a)_\lambda v=-\lambda a_\lambda v,\quad  a_\lambda (\partial v)=(\partial+\lambda)a_\lambda v, \quad \\
    &{}a_\lambda (b_\mu v)- b_\mu (a_\lambda v)=[a_\lambda b]_{\lambda+\mu} v.\quad
\end{align}
for $a,b \in\mathcal{R}, v\in M$.

If $M$ is finitely generated over $\mathbb{C}[\partial]$, then $M$ is simply called \emph {finite}. The \emph {rank} of a conformal module $M$ is its rank as a $\mathbb{C}[\partial]$-module. A conformal module $M$ is called \emph {irreducible} if it has no nontrivial submodules.
\end{em}		
\end{definition}

In the following, since we only consider conformal modules, we abbreviate ``conformal modules" into ``module".

Let $\mathcal{R}$ be a Lie conformal algebra and $M$ an $\mathcal{R}$-module. An element $m\in M$ is called \emph {invariant} if $\mathcal{R}_\lambda m=0$. Obviously, the set of all invariants of $M$ is a conformal submodule of $M$, denoted by $M^0$. An $\mathcal{R}$-module $M$ is called trivial if $M^0=M$, i.e., a module on which $\mathcal{R}$ acts trivially. For any $a \in \mathbb{C}$, we obtain a natural trivial  $\mathcal{R}$-module $\mathbb{C}_a$ which is determined by $a$, such that $\mathbb{C}_a=\mathbb{C}$ and $\partial m=am, \mathcal{R}_\lambda m=0$ for all $m \in \mathbb{C}_a$. It is easy to check that the modules $\mathbb{C}_a$ with $a \in \mathbb{C}$ exhaust all trivial irreducible $\mathcal{R}$-modules. Therefore, we only need to consider nontrivial modules in the sequel.

For any $\mathcal{R}$-module $M$, we have some basic results as follows.

\begin{lemma}\label{lem1}(Ref. \cite{KacL}, Lemma 2.2)
Let $\mathcal{R}$ be a Lie conformal algebra and $M$ an $\mathcal{R}$-module.
\begin{enumerate}
 \item If $\partial m=am$ for some $a \in \mathbb{C}$ and $m \in M$, then $\mathcal{R}_\lambda m=0$.
 \item If $M$ is a finite module without any nonzero invariant element, then $M$ is a free $\mathbb{C}[\partial]$-module.
\end{enumerate}
\end{lemma}

 Let $M$ be an $\mathcal{R}$-module. An element $m \in M$ is called a \emph {torsion element} if there exists a nonzero polynomial $ p(\partial)\in \mathbb{C}[\partial]$ such that $p(\partial)m=0$. For any $\mathbb{C}[\partial]$-module $M$, it is not difficult to check that there exists a nonzero torsion element if and only if there exists nonzero $m \in M$ such that $\partial m=am$ for some $a \in \mathbb{C}$ by using \emph {The Fundamental Theorem of Algebra}. By Lemma \ref{lem1}, we can deduce that a finitely generated $\mathbb{C}[\partial]$-module is free if and only if it has no nonzero torsion element. Thus, the following result is obvious.

 \begin{lemma}
Let $\mathcal{R}$ be a Lie conformal algebra and $M$ be a finite irreducible $\mathcal{R}$-module. Then $M$ has no nonzero torsion elements and is free of a finite rank as a $\mathbb{C}[\partial]$-module.
\end{lemma}

Similar to the definition of the $j$-th product $a_{(j)}b$ of two elements $a,b \in\mathcal{R}$, we can also define \emph {$j$-th actions} of $\mathcal{R}$ on $M$ for each $j \in\mathbb{Z^{+}}$, i.e., $a_{(j)}v$ for any $a \in\mathcal{R},v\in M$ by
\begin{align}
  a_\lambda v=\sum_{j\in\mathbb{Z^{+}}}(a_{(j)}v)\frac{\lambda^{j}}{j!}.
\end{align}
By \cite{CK}, Cheng and Kac investigated a close connection between the module of a Lie conformal algebra and that of its extended annihilation algebra.

\begin{lemma}\label{ll1}
Let $\mathcal{R}$ be a Lie conformal algebra and $M$ be a $\mathcal{R}$-module. Then $M$ is precisely a module over $Lie(\mathcal{R})^e$ satisfying the property
\begin{align}\label{eq1}
  a_{(n)}\cdot v=0,\quad n\geq N,
\end{align}
for $a \in \mathcal{R}, v \in M$, where $N$ is a non-negative integer depending on $a$ and $v$.
\end{lemma}

\begin{remark}
\begin{em}
By abuse of notations, we also call a Lie algebra module satisfying (\ref{eq1}) a conformal module over $Lie(\mathcal{R})^e$.
\end{em}
\end{remark}

\section{Finite irreducible modules over $\mathcal{W}(a,b)$}

In this section, we introduce the definition of Lie conformal algebra $\mathcal{W}(a,b)$ and consider the extended annihilation algebra $Lie(\mathcal{W}(a,b))^e$ of $\mathcal{W}(a,b)$ at first. Then we investigate the irreducibility property of the free nontrivial $\mathcal{W}(a,b)$-modules of rank one. Finally, we classify all finite nontrivial irreducible modules over $\mathcal{W}(a,b)$.

\subsection{The definition of $\mathcal{W}(a,b)$ and its extended annihilation algebra}~

\begin{definition}
The Lie conformal algebra $\mathcal{W}(a,b)$ with two parameters $a$, $b\in \mathbb{C}$ is a free $\mathbb{C}[\partial]$-module generated by $L$ and $W$ satisfying
\begin{eqnarray}
[L_\lambda L]=(\partial+2\lambda)L,~~~~[L_\lambda W]=(\partial+a\lambda+b)W,~~~~[W_\lambda W]=0.
\end{eqnarray}
\end{definition}

Note that the subalgebra $\mathbb{C}[\partial]L$ is the Virasoro Lie conformal algebra $Vir$. It is well known from \cite{CK} that

 \begin{proposition}\label{pr1}
 All free nontrivial $Vir$-modules of rank one over $\mathbb{C}[\partial]$ are as follows($ \alpha,\beta \in \mathbb{C}$):
  \begin{equation}
  \xymatrix{M_{\alpha,\beta}=\mathbb{C}[\partial]v,\qquad L_\lambda v=(\partial+\alpha\lambda+\beta)v.}
 \end{equation}
 Moreover, the module $M_{\alpha,\beta}$ is irreducible if and only if $\alpha$ is non-zero. The module $M_{0,\beta}$ contains a unique nontrivial submodule $ (\partial+\beta)M_{0,\beta}$ isomorphic to $M_{1,\beta}$. The modules $M_{\alpha,\beta}$ with $\alpha\neq 0$ exhaust all finite irreducible nontrivial $Vir$-modules.
 \end{proposition}

Next, we study the extended annihilation algebra of $\mathcal{W}(a,b)$.
\begin{lemma}\label{l21}
The annihilation algebra of $\mathcal{W}(a,b)$ is $Lie(\mathcal{W}(a,b))^+= \sum \limits_{\mathclap {m\geq-1}}\mathbb{C}L_m\oplus\sum\limits_{\mathclap{n\geq 0}}\mathbb{C}W_n$ with the following Lie bracket:
\begin{equation}\label{eq2}
  [L_m,L_n]=(m-n)L_{m+n},~[L_m,W_n]=[(a-1)(m+1)-n]W_{m+n}+bW_{m+n+1},~[W_m,W_n]=0.
\end{equation}
And the extended annihilation algebra $Lie(\mathcal{W}(a,b))^e=Lie(\mathcal{W}(a,b))^+\rtimes \mathbb{C}\partial$ satisfying (\ref{eq2}) and
\begin{equation}\label{eq3}
  [\partial,L_m]=-(m+1)L_{m-1},\quad[\partial,W_n]=-nW_{n-1}.\\
\end{equation}
\end{lemma}

\begin{proof}
By the definitions of the $j$-th product and $\mathcal{W}(a,b)$, we have
\begin{align*}
   & L_{(0)}L=\partial L,\quad L_{(1)}L=2L,\quad L_{(0)}W=(\partial+b)W,\quad L_{(1)}W=aW,\\
   & L_{(j)}L=L_{(j)}W=W_{(j-2)}W=0,\quad \forall j\geq2.
\end{align*}
Then (\ref{106}) implies for all $m$, $n$ $\in\mathbb{Z^{+}}$ that
\begin{align*}
  &[L_{(m)}, L_{(n)}]=\sum_{j \in\mathbb{Z^{+}}}\binom{m}{j}(L_{(j)}L)_{(m+n-j)}=(m-n)L_{(m+n-1)},\\
  &[L_{(m)}, W_{(n)}]=\sum_{j \in\mathbb{Z^{+}}}\binom{m}{j}(L_{(j)}W)_{(m+n-j)}=[(a-1)m-n]W_{(m+n-1)}+bW_{(m+n)},\\
  &[W_{(m)}, W_{(n)}]=0.
\end{align*}
Setting $L_m=L_{(m+1)}$ for $m\geq-1$ and $W_n=W_{(n)}$ for $n\geq0$, we obtain (\ref{eq2}) immediately. By setting $\partial=-\partial_t$, then (\ref{eq3}) follows. And the extended annihilation algebra is $Lie(\mathcal{W}(a,b))^e=Lie(\mathcal{W}(a,b))^+\rtimes \mathbb{C}\partial$.
\end{proof}

Denote $\mathcal{L}_{\mathcal{W}}(a,b)=Lie(\mathcal{W}(a,b))^e$. Define $\mathcal{L}_{\mathcal{W}}(a,b)_n=\sum \limits_{\mathclap {i\geq n}}(\mathbb{C}L_i\oplus\mathbb{C}W_i)$. Then we obtain a filtration of subalgebras of $\mathcal{L}_{\mathcal{W}}(a,b)$:
\begin{equation*}
  \mathcal{L}_{\mathcal{W}}(a,b)\supset \mathcal{L}_{\mathcal{W}}(a,b)_{-1}\supset \mathcal{L}_{\mathcal{W}}(a,b)_0\supset\cdot\cdot\cdot\supset\cdot\cdot\cdot\supset \mathcal{L}_{\mathcal{W}}(a,b)_{n}\supset\cdot\cdot\cdot,
\end{equation*}
where we set $W_{-1}=0$. Note that $\mathcal{L}_{\mathcal{W}}(a,b)_{-1}= Lie(\mathcal{W}(a,b))^{+}$.

Then we can obtain some simple results about the above filtration of subalgebras of $\mathcal{L}_{\mathcal{W}}(a,b)$.
\begin{lemma}\label{l22}
(1) For $m,n\in\mathbb{Z^+},[\mathcal{L}_{\mathcal{W}}(a,b)_{m},\mathcal{L}_{\mathcal{W}}(a,b)_{n}]\subset\mathcal{L}_{\mathcal{W}}(a,b)_{m+n}$. And $\mathcal{L}_{\mathcal{W}}(a,b)_{n}$ is an ideal of $\mathcal{L}_{\mathcal{W}}(a,b)_{0}$ for all $n\in\mathbb{Z^+}$.\\
(2) $[\partial,\mathcal{L}_{\mathcal{W}}(a,b)_{n}]=\mathcal{L}_{\mathcal{W}}(a,b)_{n-1}$ for all $n\in\mathbb{Z^+} $.
\end{lemma}

\begin{lemma}\label{l23}
(1) If $a\neq0, a\neq1, b \in \mathbb{C}, [\mathcal{L}_{\mathcal{W}}(a,b)_{0},\mathcal{L}_{\mathcal{W}}(a,b)_{0}]=\mathbb{C}W_0+\mathcal{L}_{\mathcal{W}}(a,b)_{1}$.\\
(2) If $a=1, b \in \mathbb{C}, [\mathcal{L}_{\mathcal{W}}(a,b)_{0},\mathcal{L}_{\mathcal{W}}(a,b)_{0}]=\mathcal{L}_{\mathcal{W}}(a,b)_{1}$.\\
(3) If $a=0, b \in \mathbb{C}, [\mathcal{L}_{\mathcal{W}}(a,b)_{0},\mathcal{L}_{\mathcal{W}}(a,b)_{0}]\subset\mathbb{C}W_0+\mathcal{L}_{\mathcal{W}}(a,b)_{1}$.
\end{lemma}
\begin{proof}
By the relations of $\mathcal{L}_{\mathcal{W}}(a,b)$ in Lemma \ref{l21}, we can obtain the results immediately.
\end{proof}

\begin{lemma}\label{ll3}
For a fixed $N \in\mathbb{Z^+},\mathcal{L}_{\mathcal{W}}(a,b)_0/\mathcal{L}_{\mathcal{W}}(a,b)_N$ is a finite-dimensional solvable Lie algebra.
\end{lemma}
\begin{proof}
Obviously, $\mathcal{L}_{\mathcal{W}}(a,b)_0/\mathcal{L}_{\mathcal{W}}(a,b)_N$ is a finite-dimensional Lie algebra. Therefore, we only need to prove that it is solvable. Denote the derived subalgebras of $\mathcal{L}_{\mathcal{W}}(a,b)_0$ by $\mathcal{L}_{\mathcal{W}}(a,b)_0^{(0)}=\mathcal{L}_{\mathcal{W}}(a,b)_0$ and $\mathcal{L}_{\mathcal{W}}(a,b)_0^{(n+1)}=[\mathcal{L}_{\mathcal{W}}(a,b)_0^{(n)},\mathcal{L}_{\mathcal{W}}(a,b)_0^{(n)}]$ for $n \in \mathbb{Z^{+}}$.

 No matter which case of Lemma \ref{l23}, we can get  $\mathcal{L}_{\mathcal{W}}(a,b)_0^{(1)}=[\mathcal{L}_{\mathcal{W}}(a,b)_0,\mathcal{L}_{\mathcal{W}}(a,b)_0]\subset\mathbb{C}W_0+\mathcal{L}_{\mathcal{W}}(a,b)_{1}$. Since $[W_0,\mathcal{L}_{\mathcal{W}}(a,b)_1]\subset\mathcal{L}_{\mathcal{W}}(a,b)_1$, we obtain $\mathcal{L}_{\mathcal{W}}(a,b)_0^{(2)}\subset\mathcal{L}_{\mathcal{W}}(a,b)_1$. By Lemma \ref{l22} and calculating this derived subalgebras, we get $\mathcal{L}_{\mathcal{W}}(a,b)_0^{(n+1)}\subset\mathcal{L}_{\mathcal{W}}(a,b)_n$ for all $n\in \mathbb{Z^+}$. Then we obtain this conclusion.
\end{proof}
~

\subsection{Irreducible rank one modules over $\mathcal{W}(a,b)$}~

In this section, we first give a characterization of  rank one $\mathcal{W}(a,b)$-modules.

\begin{proposition}\label{pr2}
 If $a\neq 1$ or $b\neq 0$, all free nontrivial $\mathcal{W}(a,b)$-modules of rank one over $\mathbb{C}[\partial]$ are as follows:
  \begin{equation*}
  \xymatrix{M_{\alpha,\beta}=\mathbb{C}[\partial]v,\quad L_\lambda v=(\partial+\alpha\lambda+\beta)v,\quad W_\lambda v=0,\quad for\ some\ \alpha,\beta \in \mathbb{C}.}
 \end{equation*}
 In addition, all free nontrivial $\mathcal{W}(1,0)$-modules of rank one over $\mathbb{C}[\partial]$ are as follows:
 \begin{equation*}
  \xymatrix{M_{\alpha,\beta,\gamma}=\mathbb{C}[\partial]v,\quad L_\lambda v=(\partial+\alpha\lambda+\beta)v,\quad W_\lambda v=\gamma v,\quad for\ some\ \alpha,\beta,\gamma \in \mathbb{C}.}
 \end{equation*}
 \end{proposition}
\begin{proof} Assume that
\begin{equation}
  L_\lambda v=f(\partial,\lambda)v,\qquad W_\lambda v=g(\partial,\lambda)v,
\end{equation}
 where $f(\partial,\lambda),g(\partial,\lambda)\in \mathbb{C}[\partial,\lambda]$.
 Since $\mathbb{C}[\partial]L$ is the Virasoro Lie conformal algebra, we obtain that $f(\partial,\lambda)=0$ or $f(\partial,\lambda)=\partial+\alpha\lambda+\beta$ for some $\alpha,\beta \in \mathbb{C}$ by Proposition \ref{pr1}. Since $[W_\lambda W]=0$, we have $W_\lambda(W_\mu v)=W_\mu(W_\lambda v)$. Therefore, we can obtain $g(\partial+\lambda,\mu)g(\partial,\lambda)=g(\partial+\mu,\lambda)g(\partial,\mu)$, which implies $\text{deg}_\partial g(\partial,\lambda)+\text{deg}_\lambda g(\partial,\lambda)=\text{deg}_\lambda g(\partial,\lambda)$, where $\text{deg}_\lambda g(\partial,\lambda)$ is the highest degree of $\lambda$ in $g(\partial,\lambda)$. Thus $\text{deg}_\partial g(\partial,\lambda)=0$, i.e.,  $g(\partial,\lambda)=g(\lambda)$ for some $g(\lambda) \in \mathbb{C}[\lambda]$. Finally, by $[L_\lambda W]_{\lambda+\mu}v=((\partial+a\lambda+b)W)_{\lambda+\mu}v$, we can get
 \begin{equation}\label{eq5}
   (-\lambda-\mu+a\lambda+b)g(\lambda+\mu)=g(\mu)f(\partial,\lambda)-f(\partial+\mu,\lambda)g(\mu).
 \end{equation}
 Obviously, if $f(\partial,\lambda)=0$, then $g(\lambda)=0$, which means that this module action is trivial. Therefore, $f(\partial,\lambda)=\partial+\alpha\lambda+\beta$. Taking this into (\ref{eq5}), we obtain
 \begin{equation}\label{eq6}
   ((a-1)\lambda-\mu+b)g(\lambda+\mu)=-\mu g(\mu).
 \end{equation}
 It is easy to check that $g(\mu)=\gamma$ for some $\gamma \in \mathbb{C}$. Then plugging it into (\ref{eq6}), we obtain $g(\partial,\lambda)=0$ if $a\neq1$ or $b\neq0$ and $g(\partial,\lambda)=\gamma$ for any $\gamma \in \mathbb{C}$ if $a=1$ and $b=0$.

 This completes the proof.

\end{proof}

Now we can discuss the irreducibility property of rank one $\mathcal{W}(a,b)$-modules given in Proposition \ref{pr2}.

Denote the module $M$ in Proposition \ref{pr2} by $M_{\alpha,\beta}$ (respectively, $M_{\alpha,\beta,\gamma}$) if $a\neq 1$ or $b\neq 0$ (respectively, $a=1$ and $b=0$).

\begin{proposition}\label{pr3}
(1) If $a\neq1$ or $ b\neq0$, $M_{\alpha,\beta}$ is an irreducible $\mathcal{W}(a,b)$-module if and only if $\alpha\neq0$.\\
(2) If $a=1$ and $ b=0$, $M_{\alpha,\beta,\gamma}$ is an irreducible $\mathcal{W}(a,b)$-module if and only if $\alpha\neq0$ or $\gamma\neq0$.
\end{proposition}
\begin{proof}
(1) In this case, $W_\lambda v=0$. The irreducibility of  $M_{\alpha,\beta}$ as a $\mathcal{W}(a,b)$-module is equivalent to that as a $Vir$-module. Therefore, by Proposition \ref{pr1}, we obtain the conclusion.\\
(2) In this case, $W_\lambda v=\gamma v$.
If $\alpha=\gamma=0$, then $(\partial+\beta)v$ generates a proper submodule of $M_{0,\beta,0}$, which is exactly a $Vir$-module. Thus $M_{0,\beta,0}$ is a reducible $\mathcal{W}(1,0)$-module. By the theory of $Vir$-modules, $(\partial+\beta)M_{0,\beta,0}$ is the unique nontrivial submodule of $M_{0,\beta,0}$, and $(\partial+\beta)M_{0,\beta,0}\cong M_{1,\beta,0}$.

If $\alpha=0$, and $\gamma\neq0$, we assume that $V$ is a nonzero submodule of $M_{0,\beta,\gamma}$. There exists an element $u=f(\partial)v \in V$ for some nonzero $f(\partial) \in \mathbb{C}[\partial]$. If $\text{deg}f(\partial)=0$, then $v \in V$ and $V=M_{0,\beta,\gamma}$. If $\text{deg}f(\partial)=m>0$, then $W_\lambda u=f(\partial+\lambda)(W_\lambda v)=\gamma f(\partial+\lambda)v\in V[\lambda]$. The coefficient of $\lambda^m$ in $\gamma f(\partial+\lambda)v$ is a nonzero multiple of $v$, which implies $V=M_{0,\beta,\gamma}$. Thus, $V=M_{0,\beta,\gamma}$ is an irreducible $\mathcal{W}(1,0)$-module.

If $\alpha\neq0$, $M_{\alpha,\beta,\gamma}$ is an irreducible $Vir$-module. Therefore, it is also an irreducible $\mathcal{W}(a,b)$-module.
\end{proof}

\subsection{Classification of finite irreducible modules over $\mathcal{W}(a,b)$}~

Since $\mathcal{W}(a,b)$ is of finite rank as a $\mathbb{C}[\partial]$-module, we can obtain the following result immediately.
\begin{lemma}\label{ll2}
For a conformal module $V$ over $\mathcal{W}(a,b)$ and an element $v \in V$, there exists an integer $m \in \mathbb{Z^{+}}$ such that $\mathcal{L}_{\mathcal{W}}(a,b)_m\cdot v=0$.
\end{lemma}
\begin{proof}
 By Lemma \ref{ll1}, $V$ is actually a module over $\mathcal{L}_{\mathcal{W}}(a,b)$ satisfying the following property: for each $v \in V$, there exists $N_1\geq-1$ and $N_2\geq0$ such that
 \begin{align*}
   L_m\cdot v=0, m\geq N_1,\\
   W_n\cdot v=0, n\geq N_2.
 \end{align*}
 Choosing $m=max\{N_1,N_2\}$, one can prove $L_n\cdot v= W_n\cdot v=0$, for all $n \geq m$. Thus $\mathcal{L}_{\mathcal{W}}(a,b)_m \cdot v=0$
\end{proof}

Through the above discussion, we can prove the main result of this article.\\

\begin{theorem}\label{t1}
Any finite nontrivial irreducible $\mathcal{W}(a,b)$-module M is free of rank one over $\mathbb{C}[\partial]$, and $M$ is isomorphic to $M_{\alpha,\beta}$ with $\alpha \neq 0$ if $a\neq 1$ or $b\neq 0$. In addition, any finite nontrivial irreducible $\mathcal{W}(1,0)$-module M is free of rank one over $\mathbb{C}[\partial]$, and $M$ is isomorphic to $M_{\alpha,\beta,\gamma}$ with $\alpha \neq 0$ or $\gamma \neq 0$.
\end{theorem}

\begin{proof} Assume that $V$ is a finite nontrivial irreducible $\mathcal{W}(a,b)$-module. Let $V_n=\{v\in V|\mathcal{L}_{\mathcal{W}}(a,b)_n\cdot v=0\}$. By Lemma \ref{ll2}, there exists some $n$ such that $V_n\neq \{0\}$. Let $N$ be the minimal integer such that $V_N\neq \{0\}$ and denote $U=V_N$. Since $V$ is nontrivial, we can deduce that $N\geq0$. By Lemma 3.1 in \cite{CK}, $U$ is finite dimensional.

 Since $\mathcal{L}_{\mathcal{W}}(a,b)_{N}$ is an ideal of $\mathcal{L}_{\mathcal{W}}(a,b)_{0}$, $U$ is a $\mathcal{L}_{\mathcal{W}}(a,b)_{0}$-module. Because of $\mathcal{L}_{\mathcal{W}}(a,b)_{N}\cdot U=0$, $U$ is exactly an $\mathcal{L}_{\mathcal{W}}(a,b)_{0}/\mathcal{L}_{\mathcal{W}}(a,b)_{N}$-module. By Lemma \ref{ll3}, $\mathcal{L}_{\mathcal{W}}(a,b)_{0}/\mathcal{L}_{\mathcal{W}}(a,b)_{N}$ is a finite-dimensional solvable Lie algebra. Because of \emph{Lie Theorem}, there exists a nonzero common eigenvector $v\in U$ under the action of $\mathcal{L}_{\mathcal{W}}(a,b)_{0}/\mathcal{L}_{\mathcal{W}}(a,b)_{N}$ and also the action of $\mathcal{L}_{\mathcal{W}}(a,b)_{0}$. Therefore, there exists a linear function $\chi$ on $\mathcal{L}_{\mathcal{W}}(a,b)_{0}$ such that $x\cdot v= \chi(x)v$ for all $x \in \mathcal{L}_{\mathcal{W}}(a,b)_{0}$.

Set $\mathcal{Y}=\text{span}_{\mathbb{C}}\{L_{-1},\partial\}$. Then $Lie(\mathcal{W}(a,b))^e$ has a decomposition of vector spaces
\begin{equation}
  Lie(\mathcal{W}(a,b))^e=\mathcal{Y}\oplus Lie(\mathcal{W}(a,b))_{0}.
\end{equation}
 By Poincare-Birkhoff-Witt(PBW) theorem, the universal enveloping algebra of $Lie(\mathcal{W}(a,b))^e$ is
 \begin{equation}
   U(Lie(\mathcal{W}(a,b))^e)=U(\mathcal{Y})\otimes U(Lie(\mathcal{W}(a,b))_{0}),
\end{equation}
 where $U(\mathcal{Y})=span_{\mathbb{C}}\{L_{-1}^{i}\partial^{j}| i,j\in \mathbb{Z^+}\}$, as a vector space over $\mathbb{C}$. Then we obtain
\begin{equation}\label{eqs1}
  V=U(Lie(\mathcal{W}(a,b))^e)\cdot v=U(\mathcal{Y})\cdot v=\sum_{i,j\in\mathbb{Z^+}}\mathbb{C}\partial^{i}L_{-1}^j\cdot v.
\end{equation}
Obviously, $L_{-1}\cdot v\neq0$. Otherwise, if $L_{-1}\cdot v=0$, then $V=\mathbb{C}[\partial]v$ is free of rank one, which appears to contradict the result in Proposition \ref{pr2}, i.e., $L_{-1}\cdot v=(\partial+\beta)v\neq0$ for some $\beta \in \mathbb{C}$.

Since $[\mathcal{L}_{\mathcal{W}}(a,b)_{0},\mathcal{L}_{\mathcal{W}}(a,b)_{0}]$ depends on the value of $a$ and $b$, we consider the following three cases.\\

\emph{Case 1.} $a\neq0, a\neq1, b \in \mathbb{C}$.

In this case, $[\mathcal{L}_{\mathcal{W}}(a,b)_{0},\mathcal{L}_{\mathcal{W}}(a,b)_{0}]=\mathbb{C}W_0+\mathcal{L}_{\mathcal{W}}(a,b)_{1}$ by Lemma \ref{l23}. Then $W_0 \cdot v=\chi(W_0)v=0$ and $\mathcal{L}_{\mathcal{W}}(a,b)_{1}\cdot v=\chi(\mathcal{L}_{\mathcal{W}}(a,b)_{1})v=0$. Therefore, $W_\lambda v=\sum_{i\in \mathbb{Z^+}}(W_{(j)}v)\frac{\lambda^j}{j!}
=\sum_{i\in \mathbb{Z^+}}(W_{j}v)\frac{\lambda^j}{j!}=0$ and
$\chi$ is determined by $\chi(L_0)$. We can suppose that $L_0\cdot v= \alpha v$ for some $\alpha \in \mathbb{C}$.

Let $R_\partial$ (resp. $L_\partial$) be the right (resp. left) multiplication by $\partial$ in the universal enveloping algebra of  $\mathcal{L}(a,b)$. Using $ R_\partial= L_\partial-ad_\partial$ and the binomial formula, we obtain
\begin{eqnarray}
  \mathcal{L}_{\mathcal{W}}(a,b)_{n}\partial^k &=&(R_\partial)^k \mathcal{L}_{\mathcal{W}}(a,b)_{n}=(L_\partial-ad_\partial)^k \mathcal{L}(a,b)_{n}\nonumber \\
\label{eqs2}  &=&\sum_{j=0}^k \partial^{k-j}(-ad_\partial)^j \mathcal{L}_{\mathcal{W}}(a,b)_{n}=\sum_{j=0}^k \partial^{k-j}\mathcal{L}_{\mathcal{W}}(a,b)_{n-j}
\end{eqnarray}
for $n,k \in \mathbb{Z^+}$. Since $W_\lambda v=0$ and $\mathbb{C}[\partial]W$ is an ideal of $\mathcal{W}(a,b)$, $W_\lambda V=0$ by (\ref{eqs1}) and (\ref{eqs2}). Thus, the irreducibility of $V$ as a $\mathcal{W}(a,b)$-module is equivalent to that of $V$ as a $\mathcal{W}(a,b)/\mathbb{C}[\partial]W\cong Vir$-module. Then the conclusion can be directly obtained by Proposition \ref{pr1}.


\emph{Case 2.} $a=1, b \in \mathbb{C}$.

In this case, $[\mathcal{L}_{\mathcal{W}}(a,b)_{0},\mathcal{L}_{\mathcal{W}}(a,b)_{0}]=\mathcal{L}_{\mathcal{W}}(a,b)_{1}$ by Lemma \ref{l23}. Then $L_0 \cdot v=\chi(L_0)v$, $W_0 \cdot v=\chi(W_0)v$ and $\mathcal{L}_{\mathcal{W}}(a,b)_{1}\cdot v=\chi(\mathcal{L}_{\mathcal{W}}(a,b)_{1})v=0$. We can suppose that $L_0\cdot v= \alpha v$ and $W_0\cdot v= \gamma v$ for some $\alpha,\gamma \in \mathbb{C}$.

Since $V$ is a free $\mathbb{C}[\partial]$-module of finite rank, there exists $m\in \mathbb{Z^+}\setminus \{0\}$ and $f_0(\partial),...,f_m(\partial)\in \mathbb{C}[\partial]$ such that
\begin{equation}\label{eqs3}
  f_m(\partial)\cdot (L_{-1}^m\cdot v)=\sum_{j=0}^{m-1}f_j(\partial)\cdot (L_{-1}^j\cdot v),
\end{equation}
where $\{L_{-1}^j\cdot v| j=0,1,2...,m-1\}$ is $\mathbb{C}[\partial]$-linearly independent and $f_m(\partial)\neq0$.
Let $W_0$ act on (\ref{eqs3}), noting that $W_0\cdot v=\gamma v$, $[\partial,W_0]=0$, and $[L_{-1},W_0]=bW_0$.
The left side becomes
\begin{align*}
  W_0\cdot (f_m(\partial)\cdot (L_{-1}^m\cdot v))& = f_m(\partial)\cdot (W_0\cdot(L_{-1}^m\cdot v))\\
  &=\gamma f_m(\partial)(L_{-1}-b)^m\cdot v,\\
  &=\gamma f_m(\partial)(L_{-1}^m\cdot v+\sum_{j=0}^{m-1}\binom{m}{j}L_{-1}^j(-b)^{m-j}\cdot v),\\
  &=\gamma(\sum_{j=0}^{m-1}f_j(\partial)\cdot (L_{-1}^j\cdot v)+f_m(\partial)\sum_{j=0}^{m-1}\binom{m}{j}L_{-1}^j(-b)^{m-j}\cdot v),
\end{align*}
while the right side becomes
\begin{align*}
  W_0\cdot(\sum_{j=0}^{m-1}f_j(\partial)\cdot (L_{-1}^j\cdot v))&=\sum_{j=0}^{m-1}f_j(\partial)\cdot (W_0\cdot (L_{-1}^j\cdot v))\\
  &=\gamma\sum_{j=0}^{m-1} f_j(\partial)(L_{-1}-b)^j\cdot v.
\end{align*}

Then we obtain that
\begin{equation}\label{eqs4}
  \gamma(\sum_{j=0}^{m-1}f_j(\partial)\cdot (L_{-1}^j\cdot v)+f_m(\partial)\sum_{j=0}^{m-1}\binom{m}{j}L_{-1}^j(-b)^{m-j}\cdot v)=\gamma\sum_{j=0}^{m-1} f_j(\partial)(L_{-1}-b)^j\cdot v.
\end{equation}

 If $\gamma=0$, then we can deduce that $W_\lambda v=0$, which is similar to Case 1. Thus, $V\cong M_{\alpha,\beta}$ with $\alpha\neq0$. Otherwise, by (\ref{eqs4}), we obtain that
 \begin{equation}\label{eqs5}
  \sum_{j=0}^{m-1}f_j(\partial)\cdot (L_{-1}^j\cdot v)+f_m(\partial)\sum_{j=0}^{m-1}\binom{m}{j}L_{-1}^j(-b)^{m-j}\cdot v=\sum_{j=0}^{m-1} f_j(\partial)(L_{-1}-b)^j\cdot v.
\end{equation}
 Comparing the coefficients of the $L_{-1}^{m-1}\cdot v$ term on both sides of (\ref{eqs5}), we can obtain that
 \begin{equation}
   f_{m-1}(\partial)-bmf_m(\partial)=f_{m-1}(\partial).
 \end{equation}
 Since $m\in \mathbb{Z^+}\setminus \{0\}$ and $f_m(\partial)\neq0$, then $b=0$. Obviously, $\partial-L_{-1}$ is the central element of $\mathcal{L}_{\mathcal{W}}(a,b)^e$. By Schur's Lemma, there exists some $\beta \in \mathbb{C}$ such that $(\partial-L_{-1})\cdot v= -\beta v$ for any $v \in V$. Thus $L_{-1}\cdot v=(\partial+\beta)v$ for any $v \in V$. By (\ref{eqs1}) and the irreducibility of $V$, then $V=\mathbb{C}[\partial]v$ which is free of rank one. Thus, $V\cong M_{\alpha,\beta,\gamma}$ with $\alpha\neq0$ or $\gamma\neq0$, which is irreducible by Propositions \ref{pr2} and \ref{pr3} (2).\\

\emph{Case 3.} $a=0, b \in \mathbb{C}$

In this case, $[\mathcal{L}_{\mathcal{W}}(a,b)_{0},\mathcal{L}_{\mathcal{W}}(a,b)_{0}]\subset\mathbb{C}W_0+\mathcal{L}_{\mathcal{W}}(a,b)_{1}$ by Lemma \ref{l23}. In more detail, we can obtain that $[L_0,L_n]=-nL_n, [L_0,W_m]=(-m-1)W_m+bW_{m+1}$ for all $n\geq1, m\geq0$. By Lemma \ref{ll2}, one can immediately obtain that $W_0 \cdot v=0$ and $\mathcal{L}_{\mathcal{W}}(a,b)_{1}\cdot v=0$. Then with the same discussion as Case 1, we can conclude that $V\cong M_{\alpha,\beta}$ with $\alpha\neq0$.

This completes the proof.

\end{proof}

\section{Classification of finite irreducible modules over $TSV(a,b)$ and $TSV(c)$}

In this section, we apply the methods and results in Section 3 to Lie conformal algebras $TSV(a,b)$ and $TSV(c)$ and give the classification of all  finite nontrivial irreducible conformal modules over them.

\begin{definition} (Ref. \cite{Hong})
The Lie conformal algebra  $TSV(a,b)$ with two parameters $a,b \in\mathbb{C}$ is a free $\mathbb{C}[\partial]$-module generated by $L$, $Y$ and $M$ and satisfies

 \begin{equation}
  \xymatrix{[L_\lambda L]=(\partial+2\lambda)L,\qquad [L_\lambda Y]=(\partial+a\lambda+b)Y,}
 \end{equation}
  \begin{equation}
   \xymatrix{[L_\lambda M]=(\partial+2(a-1)\lambda+2b)M,\qquad [Y_\lambda Y]=(\partial+2\lambda)M,}
 \end{equation}
  \begin{equation}
    \xymatrix{[Y_\lambda M]=[M_\lambda M]=0.}
 \end{equation}

 The Lie conformal algebra  $TSV(c)$ with a parameter $c \in\mathbb{C}$ is a free $\mathbb{C}[\partial]$-module generated by $L$, $Y$ and $M$ and satisfies
 \begin{equation}
  \xymatrix{[L_\lambda L]=(\partial+2\lambda)L,\qquad [L_\lambda Y]=(\partial+\frac{3}{2}\lambda+c)Y,}
 \end{equation}
  \begin{equation}
   \xymatrix{[L_\lambda M]=(\partial+2c)M,\qquad [Y_\lambda Y]=(\partial+2\lambda)(-\partial-2c)M,}
 \end{equation}
  \begin{equation}
    \xymatrix{[Y_\lambda M]=[M_\lambda M]=0.}
 \end{equation}
 \end{definition}
Note that $\mathbb{C}[\partial]M$ is an abelian ideal of both Lie conformal algebra $TSV(a,b)$ and $TSV(c)$. Obviously, we have $TSV(a,b)/\mathbb{C}[\partial]M\cong \mathcal{W}(a,b)$ and $TSV(c)/\mathbb{C}[\partial]M\cong \mathcal{W}(\frac{3}{2},c)$.


 Next we study the extended annihilation algebras of $TSV(a,b)$ and $TSV(c)$.
\begin{lemma}\label{lf1}
(1) The annihilation algebra of $Lie(TSV(a,b))$ is $Lie(TSV(a,b))^+= \sum \limits_{\mathclap {m\geq-1}}\mathbb{C}L_m\oplus\sum\limits_{\mathclap{p\geq-\frac{1}{2}}}\mathbb{C}Y_p\oplus\sum\limits_{\mathclap{n\geq 0}}\mathbb{C}M_n$ with the following Lie bracket:
\begin{align}
  &[L_m,L_n]=(m-n)L_{m+n},\quad[L_m,M_n]=[(2a-3)(m+1)-n]M_{m+n}+2bM_{m+n+1},\nonumber\\
  \label{eqf1}&[L_m,Y_p]=[(a-1)(m+1)-(p+\frac{1}{2})]Y_{m+p}+bY_{m+p+1},\quad[Y_p,Y_q]=(p-q)M_{p+q},\\
  &[Y_p,M_n]=[M_m,M_n]=0.\nonumber
\end{align}
And the extended annihilation algebra $Lie(TSV(a,b))^e=Lie(TSV(a,b))^+\rtimes \mathbb{C}\partial$, satisfying (\ref{eqf1}) and
\begin{equation}
  [\partial,L_m]=-(m+1)L_{m-1},\quad[\partial,M_n]=-nM_{n-1},\quad[\partial,Y_p]=-(p+\frac{1}{2})Y_{p-1}.\quad\\
\end{equation}
(2) The annihilation algebra of $Lie(TSV(c))$ is $Lie(TSV(c))^+= \sum \limits_{\mathclap {m\geq-1}}\mathbb{C}L_m\oplus\sum\limits_{\mathclap{p\geq-\frac{1}{2}}}\mathbb{C}Y_p\oplus\sum\limits_{\mathclap{n\geq 0}}\mathbb{C}M_n$ with the following Lie bracket:
\begin{align}
  &[L_m,L_n]=(m-n)L_{m+n},\quad[L_m,M_n]=[-(m+1)-n]M_{m+n}+2cM_{m+n+1},\nonumber\\
  \label{eqss2}&[L_m,Y_p]=(\frac{1}{2}m-p)Y_{m+p}+cY_{m+p+1},\  [Y_p,Y_q]=(p-q)(p+q)M_{p+q-1}+2c(q-p)M_{p+q},\\
  &[Y_p,M_n]=[M_m,M_n]=0.\nonumber
\end{align}
And the extended annihilation algebra $Lie(TSV(c))^e=Lie(TSV(c))^+\rtimes \mathbb{C}\partial$, satisfying (\ref{eqss2}) and
\begin{equation}
  [\partial,L_m]=-(m+1)L_{m-1},\quad[\partial,M_n]=-nM_{n-1},\quad[\partial,Y_p]=-(p+\frac{1}{2})Y_{p-1}.\\
\end{equation}

\end{lemma}

\begin{proof}
Through a discussion similar to Lemma \ref{l21}, we can get the above results right away.
\end{proof}

Denote $\mathcal{L}(a,b)=Lie(TSV(a,b))^e$. Define $\mathcal{L}(a,b)_n=\sum \limits_{\mathclap {i\geq n}}(\mathbb{C}L_i\oplus\mathbb{C}M_i\oplus\mathbb{C}Y_{i+\frac{1}{2}})$. Then we obtain a filtration of subalgebras of $\mathcal{L}(a,b)$:
\begin{equation*}
  \mathcal{L}(a,b)\supset \mathcal{L}(a,b)_{-1}\supset \mathcal{L}(a,b)_0\supset\cdot\cdot\cdot\supset\cdot\cdot\cdot\supset \mathcal{L}(a,b)_{n}\supset\cdot\cdot\cdot,
\end{equation*}
where we set $M_{-1}=0$. Note that $\mathcal{L}(a,b)_{-1}= Lie(TSV(a,b))^{+}$.

Then we can obtain some simple results about the above filtration of subalgebras of $\mathcal{L}(a,b)$.
\begin{lemma}\label{lf3}
(1) For $m,n\in\mathbb{Z^+},[\mathcal{L}(a,b)_{m},\mathcal{L}(a,b)_{n}]\subset\mathcal{L}(a,b)_{m+n}$. And $\mathcal{L}(a,b)_{n}$ is an ideal of $\mathcal{L}(a,b)_{0}$ for all $n\in\mathbb{Z^+}$.\\
(2) $[\partial,\mathcal{L}(a,b)_{n}]=\mathcal{L}(a,b)_{n-1}$ for all $n\in\mathbb{Z^+} $.
\end{lemma}

\begin{lemma}\label{lf2}
(1) If $a\neq0, a\neq2, a\neq\frac{3}{2}, b \in \mathbb{C}, [\mathcal{L}(a,b)_{0},\mathcal{L}(a,b)_{0}]=\mathbb{C}M_0+\mathbb{C}Y_{\frac{1}{2}}+\mathcal{L}(a,b)_{1}$.\\
(2) If $a=\frac{3}{2}, b \in \mathbb{C}, [\mathcal{L}(a,b)_{0},\mathcal{L}(a,b)_{0}]=\mathbb{C}Y_{\frac{1}{2}}+\mathcal{L}(a,b)_{1}$.\\
(3) If $a=2, b \in \mathbb{C}, [\mathcal{L}(a,b)_{0},\mathcal{L}(a,b)_{0}]=\mathbb{C}M_0+\mathcal{L}(a,b)_{1}$.\\
(4) If $a=0,b=0, [\mathcal{L}(a,b)_{0},\mathcal{L}(a,b)_{0}]=\mathbb{C}M_0+\mathbb{C}Y_{\frac{1}{2}}+\mathcal{L}(a,b)_{1}$.\\
(5) If $a=0,b \neq0, [\mathcal{L}(a,b)_{0},\mathcal{L}(a,b)_{0}]\subset\mathbb{C}M_0+\mathbb{C}Y_{\frac{1}{2}}+\mathcal{L}(a,b)_{1}$.
\end{lemma}
\begin{proof}
By the relations of $\mathcal{L}(a,b)$ in Lemma \ref{lf1}, we can obtain the results immediately.
\end{proof}

\begin{lemma}\label{lf4}
For a fixed $N \in\mathbb{Z^+}, \mathcal{L}(a,b)_0/\mathcal{L}(a,b)_N$ is a finite-dimensional solvable Lie algebra.
\end{lemma}
\begin{proof}
Obviously, $\mathcal{L}(a,b)_0/\mathcal{L}(a,b)_N$ is a finite-dimensional Lie algebra. So we only need to prove that it is solvable. Denote the derived subalgebras of $\mathcal{L}(a,b)_0$ by $\mathcal{L}(a,b)_0^{(0)}=\mathcal{L}(a,b)_0$ and $\mathcal{L}(a,b)_0^{(n+1)}=[\mathcal{L}(a,b)_0^{(n)},\mathcal{L}(a,b)_0^{(n)}]$ for $n \in \mathbb{Z^{+}}$.\\
 No matter which case of Lemma \ref{lf2}, we can get  $\mathcal{L}(a,b)_0^{(1)}=[\mathcal{L}(a,b)_0$,$\mathcal{L}(a,b)_0]\subset\mathbb{C}M_0+\mathbb{C}Y_{\frac{1}{2}}+\mathcal{L}(a,b)_{1}$. Since $[M_0,\mathcal{L}(a,b)_1]\subset\mathcal{L}(a,b)_1$,
  $[Y_{\frac{1}{2}},\mathcal{L}(a,b)_1]\subset\mathbb{C}M_0+\mathcal{L}(a,b)_1$, $[\mathbb{C}M_0,Y_{\frac{1}{2}}]=0$, we obtain $\mathcal{L}(a,b)_0^{(2)}\subset\mathbb{C}M_0+\mathcal{L}(a,b)_1$. Thus $\mathcal{L}(a,b)_0^{(3)}\subset\mathcal{L}(a,b)_1$. By Lemma \ref{lf3} and calculating this derived subalgebras, we get $\mathcal{L}(a,b)_0^{(n+2)}\subset\mathcal{L}(a,b)_n$ for all $n\in \mathbb{Z^+}$. Then we obtain this conclusion.
\end{proof}

Denote $\mathcal{L}(c)=Lie(TSV(c))^e$. Define $\mathcal{L}(c)_n=\sum \limits_{\mathclap {i\geq n}}(\mathbb{C}L_i\oplus\mathbb{C}M_i\oplus\mathbb{C}Y_{i+\frac{1}{2}})$. Then we obtain a filtration of subalgebras of $\mathcal{L}(c)$:
\begin{equation*}
  \mathcal{L}(c)\supset \mathcal{L}(c)_{-1}\supset \mathcal{L}(c)_0\supset\cdot\cdot\cdot\supset\cdot\cdot\cdot\supset \mathcal{L}(c)_{n}\supset\cdot\cdot\cdot,
\end{equation*}
where we set $M_{-1}=0$. Note that $\mathcal{L}(c)_{-1}= Lie(TSV(c))^{+}$.

\begin{lemma}
(1) For $m,n\in\mathbb{Z^+},[\mathcal{L}(c)_{m},\mathcal{L}(c)_{n}]\subset\mathcal{L}(c)_{m+n}$. And $\mathcal{L}(c)_{n}$ is an ideal of $\mathcal{L}(c)_{0}$ for all $n\in\mathbb{Z^+}$.\\
(2) $[\partial,\mathcal{L}(c)_{n}]=\mathcal{L}(c)_{n-1}$ for all $n\in\mathbb{Z^+} $.
\end{lemma}

\begin{lemma}
(1) If $c=0, [\mathcal{L}(c)_{0},\mathcal{L}(c)_{0}]=\mathbb{C}M_0+\mathbb{C}Y_{\frac{1}{2}}+\mathcal{L}(c)_{1}$.\\
(2) If $c\neq0, [\mathcal{L}(c)_{0},\mathcal{L}(c)_{0}]\subset\mathbb{C}M_0+\mathbb{C}Y_{\frac{1}{2}}+\mathcal{L}(c)_{1}$.
\end{lemma}
\begin{proof}
By the relations of $\mathcal{L}(c)$ in Lemma \ref{lf1}, we can obtain the results immediately.
\end{proof}

\begin{lemma}
For a fixed $N \in\mathbb{Z^+},\mathcal{L}(c)_0/\mathcal{L}(c)_N$ is a finite-dimensional solvable Lie algebra.
\end{lemma}
\begin{proof}
It can be obtained with a similar proof of Lemma \ref{lf4}.
\end{proof}
~

 The following  result was given in \cite{Hong}.

 \begin{proposition}\label{pr5}(Ref. \cite{Hong}, Theorem 5.4)
 (1) If $a\neq 1$ or $b\neq 0$, all free nontrivial $TSV(a,b)$-modules of rank one over $\mathbb{C}[\partial]$ are as follows:
  \begin{equation*}
  \xymatrix{M_{\alpha,\beta}=\mathbb{C}[\partial]v,\quad L_\lambda v=(\partial+\alpha\lambda+\beta)v,\quad Y_\lambda v=M_\lambda v=0,\quad for\ some\ \alpha,\beta \in \mathbb{C}.}
 \end{equation*}
 In addition, all free nontrivial $TSV(1,0)$-modules of rank one over $\mathbb{C}[\partial]$ are as follows:
 \begin{equation*}
  \xymatrix{M_{\alpha,\beta,\gamma}=\mathbb{C}[\partial]v,\quad L_\lambda v=(\partial+\alpha\lambda+\beta)v,\quad Y_\lambda v=\gamma v,\quad M_\lambda v=0,\quad for\ some\ \alpha,\beta,\gamma \in \mathbb{C}.}
 \end{equation*}
(2) All free nontrivial $TSV(c)$-modules of rank one over $\mathbb{C}[\partial]$ are as follows:
  \begin{equation*}
  \xymatrix{M_{\alpha,\beta}=\mathbb{C}[\partial]v,\quad L_\lambda v=(\partial+\alpha\lambda+\beta)v,\quad Y_\lambda v=M_\lambda v=0,\quad for\ some\ \alpha,\beta \in \mathbb{C}.}
 \end{equation*}
 \end{proposition}

Denote the module $M$ in Proposition \ref{pr5} (1) by $M_{\alpha,\beta}$ (respectively, $M_{\alpha,\beta,\gamma}$) if $a\neq 1$ or $b\neq 0$ (respectively, $a=1$ and $b=0$). Denote the module $M$ in Proposition \ref{pr5} (2) by $M_{\alpha,\beta}$. Thus, we can obtain the following proposition by using the method involved in Section 3.
\begin{proposition}\label{pr6}
(1) If $a\neq1$ or $ b\neq0$, $M_{\alpha,\beta}$ is an irreducible $TSV(a,b)$-module if and only if $\alpha\neq0$.\\
(2) If $a=1$ and $ b=0$, $M_{\alpha,\beta,\gamma}$ is an irreducible $TSV(1,0)$-module if and only if $\alpha\neq0$ or $\gamma\neq0$.
(3) For any $c\in \mathbb{C}$, $M_{\alpha,\beta}$ is an irreducible $TSV(c)$-module if and only if $\alpha\neq0$.
\end{proposition}
\begin{proof}
This proof is similar to that of Proposition \ref{pr3}.
\end{proof}
The following result shows that all finite nontrivial irreducible $TSV(a,b)$-modules and $TSV(c)$-modules are free of rank one and thus of the kind given in Proposition \ref{pr6}.
\begin{theorem}\label{t3}
(1) Any  finite nontrivial irreducible $TSV(a,b)$-module $M$ is free of rank one over $\mathbb{C}[\partial]$, and $M$ is isomorphic to $M_{\alpha,\beta}$ with $\alpha \neq 0$ (respectively, $M_{\alpha,\beta,\gamma}$ with $\alpha \neq 0$ or $\gamma \neq 0$) if $a\neq 1$ or $b\neq 0$ (respectively, $a=1$ and $b=0$).

(2) Any  finite nontrivial irreducible $TSV(c)$-module $M$ is free of rank one over $\mathbb{C}[\partial]$, and $M$ is isomorphic to $M_{\alpha,\beta}$ with $\alpha \neq 0$.
\end{theorem}
\begin{proof}
(1) Assume that $V$ is a  finite nontrivial irreducible $TSV(a,b)$-module. It is also a conformal module over $\mathcal{L}(a,b)^e$ by Lemma \ref{ll1}. By Lemmas \ref{lf1}-\ref{lf4} and using similar arguments as in the proof of Theorem \ref{t1}, we can find a nonzero vector $v$ such that
\begin{equation}
  L_i\cdot v= M_i\cdot v=Y_{i+\frac{1}{2}}\cdot v=0,\qquad \forall i \in \mathbb{Z^+}\setminus \{0\},
\end{equation}
and $L_0\cdot v=\alpha v$, $M_0\cdot v=\gamma v$ and $Y_{\frac{1}{2}}\cdot v= \xi v$ for some $\alpha,\gamma,\xi \in \mathbb{C}$. By Lemma \ref{lf2} and the same discussion as that in the proof of Theorem \ref{t1}, we can obtain that $\gamma=0$, i.e., $M_0\cdot v=0$ in Cases (1),(3),(4) and (5) in Lemma \ref{lf2}. In Case (2) in Lemma \ref{lf2}, if $\gamma\neq0$, then we do the discussion similar to that in the proof of Case 2 in Theorem \ref{t1}. But in this case, the basis of
\begin{equation}
  V=U(Lie(\mathcal{L}(a,b))^e)\cdot v=\sum_{i,j,k\in\mathbb{Z^+}}\mathbb{C}\partial^{i}L_{-1}^jY_{-\frac{1}{2}}^k\cdot v.
\end{equation}
Then we can deduce $b=0$ by the similar discussion to the proof of Case 2 in Theorem \ref{t1}. If $\gamma=0$, we can deduce that $V$ is also a $TSV(a,b)/\mathbb{C}[\partial]M\cong \mathcal{W}(a,b)$-module. Thus, we can obtain that $V\cong M_{\alpha,\beta}$ or $V\cong M_{\alpha,\beta,\gamma}$ which depends on the value of $a,b$. If $\gamma\neq0, b=0$ in Case (2) of Lemma \ref{lf2}, we can obtain $\gamma=0$ by using the same method in the proof of Theorem 4.3 in \cite{Wu-Yuan}. Thus $\gamma=0$ for all cases in Lemma \ref{lf2}. Then the conclusion can be obtained by Theorem \ref{t1}.

(2) Similar discussion to the proof of (1), we can deduce that $TSV(c)/\mathbb{C}[\partial]M\cong \mathcal{W}(\frac{3}{2},c)$ and $V \cong M_{\alpha,\beta}$.

This completes the proof.

\end{proof}

\begin{remark}
Note that $TSV(0,0)$ is just the Schr\"{o}dinger-Virasoro type Lie conformal algebra studied in \cite{Wang-Xu-Xia}. Theorem \ref{t3} also gives a characterization of all finite nontrivial modules of the Schr\"{o}dinger-Virasoro type Lie conformal algebra in \cite{Wang-Xu-Xia}.
\end{remark}


%
\end{document}